\documentclass[11pt]{amsart}

\usepackage{amsmath, amsthm, amssymb}
\usepackage{url}
\usepackage{graphicx}
\usepackage{color}
\usepackage{amssymb,amsfonts}
\usepackage[font={small,it}]{caption}
\usepackage[flushleft]{threeparttable}
\usepackage{stmaryrd}
\usepackage{mathrsfs}
\usepackage{soul}

\usepackage[top=3cm, bottom=3cm, left=2.5cm, right=3cm]{geometry}

\theoremstyle{italic}

\newtheorem{theorem}{{\sc Theorem}}
\newtheorem{defn}[theorem]{{\sc Definition}}

\newtheorem{lemma}[theorem]{Lemma}

\newtheorem{remark}[theorem]{Remark}
\newtheorem{remarks}[theorem]{Remarks}

\theoremstyle{definition}

\newcommand{\bfX}{\mathbf{X}}

\newcommand{\bbX}{\mathbb{X}}

\newcommand{\R}{\mathbb{R}}

\newcommand{\E}{\mathbb{E}}

\newcommand{\bbA}{\mathbb{A}}

\newcommand{\OO}{\mathbb{O}}

\newcommand{\mcC}{\mathcal{C}}
\newcommand{\mcA}{\mathcal{A}}
\newcommand{\mcU}{\mathcal{U}}
\newcommand{\mcV}{\mathcal{V}}

\newcommand{\bfa}{\bf a}

\newcommand{\Lip}{\operatorname{Lip}}

\newcommand{\spann}{\operatorname{span}}

\newcommand{\drp}{\ell} % dimension of rough path
\newcommand{\dode}{d}   % dimension of ODE solution

\newenvironment{Dem}{%
    \begin{list}{\hspace{0.5cm}{\sc Proof --}}{%
        \setlength{\topsep}{0pt}%
        \setlength{\leftmargin}{0pt}%
        \setlength{\rightmargin}{0pt}%
        \setlength{\listparindent}{0pt}%
        \setlength{\itemindent}{0pt}%
        \setlength{\parsep}{0pt}%
        \addtolength{\leftmargin}{20pt}%
        \addtolength{\rightmargin}{0pt}%
    } \item }{\hfill{\space $\rhd$}\end{list}\smallskip}

\begin{document}

\title{The inverse problem for rough controlled differential equations}
\author{I. Bailleul and J. Diehl}
\address{IRMAR, 263 Avenue du General Leclerc, 35042 RENNES, France}
\email{ismael.bailleul@univ-rennes1.fr}
\address{Institut f\"ur Mathematik, MA 7-2. Strasse des 17. Juni 136 10623 BERLIN, Germany}
\email{diehl@math.tu-berlin.de}
\thanks{This research was partially supported by the program ANR "Retour Post-Doctorants", under the contract ANR 11-PDOC-0025, and DFG Schwerpunktprogramm 1324. The first author thanks the U.B.O. for their hospitality. The second author would like to thank Paul Gassiat and Nicolas Perkowski for helpful remarks.}

\maketitle

\begin{abstract}
We provide a necessary and sufficient condition for a rough control driving a differential equation to be reconstructable, to some order, from observing the resulting controlled evolution. Physical examples and applications in stochastic filtering and statistics demonstrate the practical relevance of our result.
\end{abstract}

% \tableofcontents
% \newpage

%------------------------%
\section{Introduction}
\label{SectionIntroduction}
%------------------------%

It is a classical topic in control theory to consider differential equations of the form
\begin{align*}
  \dot x_t = \sum_{i=1}^\ell V_i(x_t) u^i_t,
\end{align*}
with nice enough vector fields $V_i: \R^\dode \to \R^\dode$, say bounded and Lipschitz for the moment, and the controls $u^i: [0,T] \to \R, i = 1, \dots, \drp$ are integrable functions that represent data that can usually be tuned on demand in a real-life problem. It will pay off to consider the controls $u^i$ as the time-derivative of absolutely continuous functions $X^i$, and reformulate the above dynamics under the form
\begin{align}
  \label{eq:controlledEquation}
  dx_t = V_i(x_t) dX^i_t,
\end{align}
with the usual convention for sums. In this work we consider the \emph{inverse problem}, of recovering the control $X$ using only the knowledge of the dynamics generated by the above equation, and assuming that we know the vector fields $V_i$.
%in the situation where the control $X$ is not smooth enough to make classical sense of the above differential equation.

With a view to applications involving noisy signals such as Brownian paths, we are particularly interested in the setting where $X$ is \emph{not} absolutely continuous, but only $\alpha$-H\"older continuous, for some $\alpha \in \big(\frac{1}{3},\frac{1}{2}\big]$, say. Although equation \eqref{eq:controlledEquation} seems non-sensical in this case, it is given meaning by the theory of rough paths, as invented by T. Lyons \cite{Lyons98}, and reformulated and enriched by numerous works since its introduction. We recall the essentials of this theory in Section \ref{SectionRoughPaths} so as to make it accessible to a large audience; it suffices for the moment to emphasize that driving signals in this theory do not consist only of $\R^\ell$-valued $\alpha$-H\"older continuous paths $X$, rather they are pairs $(X,\bbX)$, where the $\R^\ell\times\R^\ell$-valued second object $\mathbb{X}$, is to be thought of as the collection of ''iterated integrals $\int X^j dX^k$''. The point here is that the latter expression is a priori meaningless if we only know that $X$ is $\alpha$-H\"older, for $\alpha\leq \frac{1}{2}$, as one cannot make sense of such integrals without additional structure. So these iterated integrals have to be provided as \emph{additional a priori data}. The enriched control $(X,\mathbb{X})$ is what we call a {\bf\textit{rough path}}. Lyons' main breakthrough in his fundamental work \cite{Lyons98} was to show that equation \eqref{eq:controlledEquation} can be given sense, and is well-posed, whenever $X$ is understood as a rough path, provided the vector fields $V_i$ are sufficiently regular. Moreover the solution map which associates to the starting point $x_0$, and the rough path $(X,\mathbb{X})$, the solution path $x_\bullet$ is continuous. This is in stark contrast with the fact that solutions of stochastic differential equations driven by Brownian motion are only measurable functions of their drivers, with no better dependence as a rule. Section \ref{SectionRoughPaths} offers a short review of rough paths theory sufficient to grasp the core ideas of the theory and for our needs in this work.

From a practical point of view, it may be argued that there are no physical examples of real-life observable paths which are truely $\alpha$-H\"older continuous paths; rather, all paths can be seen as smooth, albeit with possibly very high fluctuations that make them appear rough on a macroscopic scale. However, the invisible microscopic fluctuations might lead to macroscopic effects when the path (the control) is acting on a dynamical system. A rough path somehow provides a mathematical abstraction of this fact by recording this microscopic scale effects that a highly-oscillatory smooth signal may have on a dynamics system in the second order object $\bbX$. 

\medskip

While rough paths theory could originally be thought of as a theoretical
framework for the study of controlled systems, its core notions and tools have
proved extremely useful in handling a number of practical important problems.
As an example, Lyons and Victoir's cubature method on Wiener space
\cite{LyonsVictoir} has led to efficient numerical schemes for the simulation
of various partial differential equations, such as the HJM or CIR equations,
and related quantities of interest in mathematical finance. Let us
mention, as another example, that the use of the core concept of signature of a
signal in the setting of learning theory is presently being investigated
\cite{LevinLyonsNi,GyurkoLyonsKontkowskiField}, and may well bring deep
insights into this subject. In a different direction, in \cite{CrisanDiehlOberhauser},
it was shown that the optimal filter in stochastic filtering, which is in general \emph{not} a continuous function on path space, can actually be defined as a continuous functional on the rough path space; see Section \ref{SectionApplications} for details. This result motivated the present work, since it leaves the practitioner with the task of ''observing a rough path'', if she/he wants to use this continuity result to provide robust approximations for the optimal filter, by feeding her/his approximate rough path into the continuous function that gives back the optimal filter. 
It should be clear from the above heuristic picture of a rough path that the only way to uncover the 'microscopic' second level component of a rough path is to let that rough path act on a dynamical system, a rough differential equation. 

After introducing the reader with the essentials of rough paths theory in section \ref{SectionRoughPaths}, we give as our main result, section \ref{SectionMainResult}, a necessary and sufficient condition for the reconstruction of a rough path to be possible from the observation of the dynamics generated by a rough differential equation with that rough path as driver. Its proof is given in section \ref{SectionProofs}. Section \ref{SectionExamples} provides several examples of physical systems that satisfy our assumptions, and section \ref{SectionApplications} offers applications of our result to problems in filtering and statistics.

%-------------------------------------------------------------%
\section{Rough controlled differential equations}
\label{SectionRoughPaths}
%-------------------------------------------------------------%

As a first step in rough paths theory, let us consider the controlled ordinary differential equation
$$
dx_t = V_i(x_t)dX^i_t
$$
driven by a smooth (or bounded variation) $\R^\ell$-valued control $X$. It is elementary to iterate the formula
\begin{align*}
  V_i(x_t)
  = V_i(x_s) + \int_s^t \partial_j V_i(x_r) V^j_k(X_r) dX^k_r
  =: V_i(x_s) + \int_s^t V_k V_i(x_r)dX^k_r,
\end{align*}
to obtain a kind of Taylor-Euler expansion of the solution path to the above equation, under the form
\begin{align*}
  x_t
  &=
  x_s
  +
  \sum_{n=1,\dots,N-1} \sum_{i_1, \dots, i_n = 1,\dots,\drp} \left( \int_s^t \int_s^{r_n} \dots \int_s^{r_2} dX^{i_1}_{r_1} \dots dX^{i_n}_{r_n} \right) V_{i_1} \dots V_{i_n}(x_s)  \\
  &\qquad\; + 
  \sum_{i_1, \dots, i_N = 1,\dots,\drp} \int_s^t \int_s^{r_N} \dots \int_s^{r_2} V_{i_1} \dots V_{i_N}(x_{r_1}) dX^{i_1}_{r_1} \dots dX^{i_n}_{r_N},
\end{align*}
provided the vector fields are sufficiently regular for this expression to make sense. If $V$ and its first $N$ derivatives are bounded, the last term is of order $|t-s|^N$. So the following numerical scheme 
\begin{equation}
  \label{eq:eulerScheme}
  \begin{cases}
  x^n_0 &:= x_0 \\
  x^{n}_{t_i} &:= x^n_{t_{i-1}}
  +
  \sum_{n=1,\dots,N-1} \sum_{i_1, \dots, i_n = 1,\dots,\drp}  V_{i_1} \dots V_{i_n}(x^n_{t_{i=1}}) \int_s^t \int_s^{r_n} \dots \int_s^{r_2} dX^{i_1}_{r_1} \dots dX^{i_n}_{r_n},
  \end{cases}
\end{equation}
written here for any partition $\tau = (t_i)$ of $[0,T]$, is of order $|\tau|^{N-1}$, where $|\tau|$ denotes the meshsize of the partition. See for example Proposition 10.3 in \cite{bibFrizVictoir} for more details.

\smallskip

In the mid-90's T. Lyons \cite{Lyons98} understood that one can actually make sense of the controlled differential equation \eqref{eq:controlledEquation} even if $X$ is not of bounded variation, \emph{if one provides a priori the values of sufficiently many iterated integrals $\int \dots \int dX \dots dX$, and one defines a solution path to the equation as a path for which the above scheme is exact up to a term of size $\big|t_{i+1}-t_i\big|^a$, for some constant $a>1$.} 

\begin{defn}
  Fix a finite time horizon $T$, and let $\alpha \in \big(\frac{1}{3},\frac{1}{2}\big]$. An \textbf{$\alpha$-H\"older rough path} $\mathbf{X} = (X,\mathbb{X})$ consists of a pair $(X,\bbX)$, made up  of an $\alpha$-H\"older continuous path $X: [0,T] \to \R^\drp$,
  and a two parameter function $\mathbb{X}: \{(t,s)\,;\,0\leq s\leq t\leq T\} \to (\R^\ell)^{\otimes 2}\simeq \R^{\drp \times \drp}$,
  such that the inequality
  \begin{align*}
    |\mathbb{X}_{s,t}| \le C |t-s|^{2\alpha} 
  \end{align*}
  hold for some positive constant $C$, and all $0\leq s\leq t\leq T$, and we have 
  \begin{equation}
  \label{EqChen}
    \mathbb{X}^{jk}_{ts} - \mathbb{X}^{jk}_{tu} - \mathbb{X}^{jk}_{us} = X^j_{s,u} X^k_{u,t},
  \end{equation}
  for all $0\leq s\leq u\leq t\leq T$. The rough path $\bfX$ is said to be \textbf{weakly geometric} if the symmetric part of $\bbX_{ts}$ is given in terms of $X_{ts}$ by the relation
  \begin{align*}
  \mathbb{X}^{jk}_{s,t} + \mathbb{X}^{kj}_{s,t} = X^j_{s,t}\,X^k_{s,t}. 
  \end{align*}
We define a norm on the set of $\alpha$-H\"older rough path setting
$$
\|{\bfX}\| := \|X\|_\infty + \|X\|_\alpha + \|\bbX\|_{2\alpha},
$$
where $\|\cdot\|_\gamma$ stands for the $\gamma$-H\"older norm of a 1 or 2-index map, for any $0<\gamma<1$. 
\end{defn}

\smallskip

To make sense of these conditions, think of $\bbX^{jk}_{ts}$ as $\int_s^t (X_r-X_s)^j\,dX^k_r$, even though this integral does not make sense in our setting. (Young integration theory \cite{Young,LyonsStFlour} can only make sense of such integrals if $X$ is $\alpha$-H\"older, with $\alpha >\frac{1}{2}$.) When $X$ is smooth its increments have size $(t-s)$, and the increments $\bbX_{ts}$ have size $(t-s)^2$. Convince yourself that relation \eqref{EqChen} comes in that model setting from Chasles' relation $\int_s^u + \int_u^t = \int_s^t$. The symmetry condition satisfied by weak geometric rough paths is satisfied by the rough paths lift of any smooth path. We invite the reader to check that if $B$ stands for a Brownian motion and $\frac{1}{3}\leq \alpha <\frac{1}{2}$, we define an $\alpha$-H\"older rough path setting
$$
{\bf B}_{ts} = \left(B_t-B_s,\int_s^t (B_r-B_s)\otimes dB_r\right),
$$
where the above integral is an It\^o integral. This rough path is not weakly geometric however, while we would define a weakly geometric rough path by using a Stratonovich integral in the above definition of $\bf B$. See for example Chapter 3 in \cite{FrizHairer} for more background in this direction. 

Let denote by $\sf V$ the collection of some vector fields $\big(V_1,\dots,V_\ell\big)$ on $\R^d$.
\begin{defn}
\label{DefnRDE}
  Fix a finite time horizon $T$, and let $\mathbf{X}$ be a weak geometric $\alpha$-H\"older rough path, with $\alpha \in \big(\frac{1}{3},\frac{1}{2}\big]$. A \textbf{path} $x_\bullet : [0,T] \to \R^\dode$ is said to \textbf{solve the rough differential equation}
  \begin{align}
    \label{eq:RDE}
    dx_t = {\sf V}(x_t) \mathbf{X}(dt),
  \end{align}
  if %for any function $f\in\mcC^3_b$,
  we have the following Taylor-Euler expansion
  \begin{align}
    \label{eq:eulerApproximation}
    %f(x_t) = f(x_s) + X^i_{ts}\big(V_if\big)(x_s) + \mathbb{X}^{jk}_{ts}\big(V_j V_kf\big)(x_s) + O\big(|t-s|^{3\alpha}\big),
    x_t = x_s + V_i(x_s) X^i_{ts} + V_j V_k(x_s) \mathbb{X}^{jk}_{ts} + O\big(|t-s|^{3\alpha}\big),
  \end{align}
with 
$$
  \Big|O\big(|t-s|^{3\alpha}\big)\Big| \leq c({\bfX,V})\,|t-s|^{3\alpha},
$$
for some positive constant $c({\bfX},\sf V)$ depending only on $\bfX$ and the $\Lip^3$ norm of the $\sf V$.
\end{defn}
As a sanity check, one can easily verify that if the vector fields $V_i$ are constant, the solution path to the above rough differential equation, started from $x_0$, is given by just $x_t = x_0 + X^i_{ts}V_i$.

It can be proved that if $(X,\bbX)$ is the rough path above Brownian motion introduced above, but with a Stratonovich integral rather than an It\^o integral, then a solution path to the above rough differential equation is a solution path to the Stratonovich differential equation
$$
dx_t = V_i(x_t)\,{\circ dB^i_t}.
$$
This is what makes rough paths theory so appealing for applications to stochastic calculus. See e.g. \cite{FrizHairer}.

\smallskip

Rather than solving the rough differential equation \eqref{eq:RDE} for each fixed starting point, we can actually construct a flow of maps $(\varphi_{ts})_{0\leq s\leq t\leq T}$ with the property that the path $\big(\varphi_{\bullet 0}(x)\big)$ is for each starting point $x$ the solution to equation \eqref{eq:RDE} started from $x$. Given a bounded Lipschtiz continuous vector field $W$ on $\R^d$ we denote by $\exp(W)$ the time 1 map of $W$, that associates to any point $x\in\R^d$ the value at time 1 of the solution to the ordinary differential equation $\dot y = W(y)$, started from $x$. Fix $0\leq s \leq t\leq T$, and denote by $\bbA_{ts}$ the antisymmetric part of $\bbX_{ts}$. Setting
$$
\mu_{ts} := \exp\left(\sum_{i=1}^\ell X^i_{ts}V_i + \sum_{1\leq j<k\leq \ell}\bbA^{jk}_{ts}\big[V_j,V_k\big]\right),
$$
an elementary Taylor-Euler expansion \cite{BailleulFlows}, using the weak geometric character of $\bfX$, shows that $\mu_{ts}$ satisfies the kind of Taylor-Euler expansion we expect from a solution path to the above rough differential equation, as we have
$$
\Big| f\circ\mu_{ts} - \Big(f + X^i_{ts}\,V_if + \mathbb{X}^{jk}_{ts}\,V_j V_kf\Big)\Big| \leq c({\bfX},f)\,|t-s|^{3\alpha}.
$$
Theorem \ref{ThmLyonsUniversal} is the cornerstone of the theory of rough differential equations. It was first proved in a different form by Lyons \cite{Lyons98}, and was named '\textit{Lyons' universal limit theorem}' by Malliavin. Its present form is a mix of Lyons' original result, Davie's approach \cite{Davie} and the first author's approach \cite{BailleulFlows} to rough differential equations. A vector field of class $\mcC^2_b$,with Lipschitz second derivative is said to be $\Lip^3$ in the sense of Stein.

\begin{theorem}[Lyons' universal limit theorem]
\label{ThmLyonsUniversal}
  Let $\mathbf{X}$ be an $\alpha$-H\"older rough path, $\alpha \in \big(\frac{1}{3},\frac{1}{2}\big]$.
  Let ${\sf V} = (V_i)_{i=1,\dots,\ell}$ be a collection of $\Lip^3$ vector fields on $\R^d$.
  \begin{enumerate}
      \item  There exists a unique flow of maps $(\varphi_{ts})_{0\leq s\leq t\leq T}$ on $\R^d$ such that the inequality
      $$
      \big\| \varphi_{ts} - \mu_{ts}\big\|_\infty \leq c({\bfX})\,|t-s|^{3\alpha}      
      $$
      holds for some positive constant $c({\bfX})$ depending only on $\bfX$, for all times $0\leq s\leq t\leq T$. 
      
      \item  Given any starting point $x\in\R^d$, there exists a unique solution path $x_\bullet$ to the rough differential equation \eqref{eq:RDE}; this solution path is actually given, for all $0\leq t\leq T$, by 
      $$
      x_t = \varphi_{t0}(x_0).
      $$
      \item The solution path $x_\bullet \in \Big(\mcC\big([0,T],\R^d\big),\|\cdot\|_\infty\Big)$ depends \emph{continuously} on $\bfX$.  
\end{enumerate}    
\end{theorem}

The crucial point in the above statement is the continuous dependence of the solution path as a function of the rough signal $\bfX$, in stark contrast with the fact that solutions of stochastic differential equations are only measurable functionals of the Brownian path, while rough differential equations can be used to solve Stratonovich differential equations. The twist here is that theonly purely measurable operation that is done here is in defining the iterated integrals $\int_s^t (B_r-B_s)\otimes {\circ dB_r}$; once this is done, the machinery for solving the rough differential equation \eqref{eq:RDE} is continuous with respect to the Brownian rough path. This continuity result was used for instance to give streamlined proofs of deep results in stochastic analysis such as Stroock-Varadhan support theorem for diffusion processes, or the basics of Freidlin-Wentzell theory of large deviations for diffusion processes \cite{LedouxQianZhang}.

\smallskip

Parts (2) and (3) of this theorem can be proved in several ways.The original approach of Lyons \cite{Lyons98} was to recast it under a fixed point problem involving a rough integral, which has to be defined first. (This argument has been streamlined by Gubinelli in \cite{Gubinelli}, see also the monograph \cite{FrizHairer}.) Existence and well-posedness results using second-order Milstein type scheme of the form \eqref{eq:eulerScheme} were introduced in that setting by Davie in \cite{Davie}, and generalized in the work \cite{BailleulFlows} of the first author to deal with rough differential equations driven by weak geometric $\alpha$-H\"older rough paths, for any $0<\alpha<1$, using a geometric approach with roots in the work \cite{Strichartz} of Strichartz and the novel tool of approximate flows. Part (1) of the above theorem is from \cite{BailleulFlows}.

\smallskip

We refer the reader to the textbooks \cite{LyonsQian} of Lyons and Qian, or \cite{FrizVictoir} of Friz and Victoir, for a thorough account of the theory, and to the lecture notes \cite{CaruanaLevyLyons, FrizHairer, M2Course} for shorter pedagogical accounts.

%$---------------------------------%
\section{The inverse problem}
\label{SectionMainResult}
%$---------------------------------%

We would like to propose the examples of application of rough paths theory given above, and many others, as an illustration of one of T. Lyons' leitmotivs: \textit{Rough paths are not mathematical abstractions, they appear in Nature.} Starting form this postulate, and keeping in mind that rough paths can be understood as a convenient mathematical setting for describing both the macroscopic and microscopic scales of highly oscillating signals, the aim of this work is to answer an important question that comes with this postulate: {\it Can one \emph{observe and record} a rough path?} We shall handle this problem in the model setting of a physical system associated with a rough differential equation, which leads to the following question. {\it Under what conditions on the driving vector fields can one recover the driving rough path by observing the solution flow to that equation?} It is indeed not always possible to reconstruct the driving signal; as an example, take a rough differential equation with constant vector fields, where the second level of the rough path has no influence on the solution, as made clear after definition \ref{DefnRDE}. 

To give a motivating example, assume that two dynamics are described by two rough differential equations \textit{driven by the same rough path}. Think for instance to two multi-dimensional assets. It may happen that one is interested in one of these assets while one can only observe the other. How should we proceed then if one wants to make a trade only when the unobservable asset is in some given region of its state space? If we could reconstruct approximatively the rough signal from the observation of the first asset dynamics, we could use the continuity statement in Lyons' universal limit theorem, theorem \ref{ThmLyonsUniversal}, to plug this approximate signal into the dynamics of the second asset and get an approximate path whose distance to the true second asset path is quantifiable, leading to a trading strategy. Although we whall not develop this example of optimal stopping problem with incomplete information, we shall give other examples where the approximate (or ideally, exact) reconstruction of a rough signal allows to compensate an a priori lack of information.

\medskip

To be more specific, our problem reads as follows. Given some sufficiently regular vector-field valued $1$-form $\mathsf{V}= (V_1,\dots,V_\ell)$ on $\R^\ell$, and a weak geometric $\alpha$-H\"older rough path $\bfX$ over $\R^\ell$, with $\alpha \in \big(\frac{1}{3},\frac{1}{2}\big]$, defined on the time interval $[0,1]$ say, denote by $\big(\varphi_{ts}\big)_{0\leq s\leq t\leq 1}$ the solution flow \cite{BailleulFlows} to the rough differential equation 
\begin{equation}
\label{EqRDE}
dx_r = \mathsf{V}(x_r)\,{\bfX}(dr)
\end{equation}
in $\R^d$. (This equation is the correct form that equation \eqref{eq:controlledEquation} takes when the control $h$ is a rough path $\bfX$.
Not only can it be solved for each fixed initial condition, but it also defines a flow of maps, as we have seen in Section \ref{SectionRoughPaths}. The map $\varphi_{ts}$ associates to $x$ the solution at time $t$ of equation \eqref{EqRDE} started at time $s$ from $x$.) Assume we observe increments of the different solution paths, started from $c$ distinct points $x_j,\ j=1,\dots,c$; that is, we have access to the data
$$
  z^{x_1, \dots, x_c}_{ts}  := \Big( \varphi_{ts}\big(x_1\big), \dots, \varphi_{ts}\big(x_c\big) \Big).
$$ 
Our goal is to reconstruct the driving signal $\bfX$ using uniquely this information. As the counter-example of constant vector fields shows, the ability to do so depends on the $1$-form $\mathsf{V}$.

We first make precise, what we mean by saying that ``reconstruction is possible''.
Fix $\alpha \in (1/3,1/2]$ for the rest of the paper.

\begin{defn}
  \label{DefnReconstruction}
  The $1$-form $\mathsf{V}$ is said to have the {\bf reconstruction property}
  if one can find an integer $c \ge 1$,
  points $x_1, \dots, x_c \in \R^d$,
  a constant $a>1$, and a function $\mathcal{X} : \R^{c d} \rightarrow T^2(\R^\ell) \cong \R^d\oplus(\R^d)^{\otimes 2}$, with components $\mathcal{X}^1$ and $\mathcal{X}^2$, such that one can associate to every positive constant $M$ another positive constant $C_M$ such that the inequalities
\begin{equation}
\label{EqConditionReconstruction}
    \Big|\mathcal{X}^1\big(z^{x_1, \dots, x_c}_{ts}\big) - X_{ts}\Big| \leq C_M|t-s|^a,\quad \Big|\mathcal{X}^2\big(z^{x_1, \dots, x_c}_{ts}\big) - \bbX_{ts}\Big| \leq C_M|t-s|^a
\end{equation}
  hold for {\em all} weak geometric $\alpha$-H\"older rough paths $\bfX$ with $\|\bfX\|_\alpha \le M$, for all times $0\leq s\leq t\leq 1$ sufficiently close.
\end{defn}

\noindent These inequalities ensure that the $T^2(\R^\ell)$-valued functional $\mathcal{X}\big(z^{x_1, \dots, x_c}_{ts}\big)$ is almost-multiplicative \cite{LyonsStFlour}, with associated multiplicative functional $\bfX$. Hence, by a fundamental result of Lyons \cite{Lyons98,LyonsQian}, one can - in principle - completely reconstruct $\bfX_{ts}$ from the knowledge of the $\mathcal{X}\big(z^{x_1, \dots, x_c}_{ba}\big)$, with $s\leq a\leq b\leq t$.

\medskip

\begin{remark}
  \label{rem:lieRoughPath}
  Since $\bfX$ is weak-geometric, the symmetric part of $\bbX_{ts}$ is equal to $\frac{1}{2} X_{ts} \otimes X_{ts}$. So the essential information in the rough path $\bfX$ is given by $X_{ts}$ and the antisymmetric part $\bbA_{ts}$ of $\bbX_{ts}$. This pair lives in $\R^{\frac{\ell(\ell+1)}{2}}$. For the reconstruction property to hold one can alternatively find a function $\mathscr{R}$ such that
  \begin{align*}
    \Bigl| \mathscr{R}( z^{x_1, \dots, x_c}_{ts} ) - (X_{ts}, \bbA_{ts}) \Bigr| \le C |t-s|^a;
  \end{align*}
this is actually what we shall do in the proof of the main theorem below.
\end{remark}

Our main result takes the form of a sufficient and necessary condition on the $1$-form $\mathsf{V}$ for equation \eqref{EqRDE} to have the reconstruction property. Only brackets of the form $\big[V_j,V_k\big]$, with $j<k$, appear in the matrix below.

\begin{theorem}[Reconstruction]
\label{ThmReconstruction}
Let $\mathsf{V}= (V_1,\dots,V_\ell)$ be a $\textrm{\emph{Lip}}^3(\R^d)$-valued $1$-form on $\R^\ell$.%
Set 
$$
  m:= \frac{\ell(\ell+1)}{2}.
$$
Then equation \eqref{EqRDE} has the reconstruction property if and only if there exists an integer $c$ and points $x_1,\dots, x_c$ in $\R^d$ such that the $(cd\times m)$ matrix
\begin{align*}
{\sf M} =  \begin{pmatrix}
      V_1(x_1) & \cdots & V_\ell(x_1) & [V_1,V_2](x_1) & \cdots &  [V_{\ell-1},V_\ell](x_1)  \\
      \vdots      &              & \vdots          & \vdots                  &            &  \vdots   \\
      V_1(x_c) & \cdots & V_\ell(x_c) & [V_1,V_2](x_c) & \cdots &  [V_{\ell-1},V_\ell](x_c)  
    \end{pmatrix}
\quad
  \end{align*}
has rank $m$. In this case $a$ in the definition of the reconstruction property can be chosen to be equal to $3 \alpha$.
We call ${\sf M}$ the \bf{reconstruction matrix}.
\end{theorem}

The above rank condition will hold for instance if $\ell=2$ and $\big(V_1,V_2,\big[V_1,V_2\big]\big)$ forms a free family at some point -- for which we need $d\ge 3$. One can actually prove, by classical \emph{transversality arguments}, that if $x_1,\dots,x_m$ are any given family of $m=\frac{\ell(\ell+1)}{2}$ distinct points in $\R^d$, then the set of tuples $\big(V_1,\dots,V_\ell\big)$ of $\textrm{Lip}^3$-vector fields on $\R^d$ for which the reconstruction matrix has rank $m$ is dense in $\big(\textrm{Lip}^3\big)^\ell$. This means that one can always reconstruct the rough signal in a ``generic'' rough differential equation from observing its solution flow at no more than $m$ points. (See the books by Hirsch \cite{Hirsch} or Zeidler \cite{Zeidler} for a gentle introduction to transversality-type arguments.) This genericity result obviously does not mean that any tuple $\big(V_1,\dots,V_\ell\big)$ of vector fields enjoys that property, as the above example with the constant vector fields corresponding to the canonical basis shows.

\medskip

{\sc Examples.} Here are a few illustrative examples where Theorem \ref{ThmReconstruction} applies and reconstruction is possible.

\begin{enumerate}
   \item Note that the above condition on the reconstruction matrix is unrelated to H\"ormander's bracket condition, and that there is no need of any kind of ellipticity or hypoellipticity for Theorem \ref{ThmReconstruction} to apply. If a $\textrm{Lip}^3(\R^d)$-valued $1$-form on $\R^\ell$ has the reconstruction property, its trivial extension $\widetilde{\sf V} = \big(\widetilde{V}_1,\dots,\widetilde{V}_\ell\big)$ to vector fields on $\R^{d+1}\simeq\R^d\times\R$, with the $\widetilde{V}_i = (V_i,0)$, does not involve a hypoelliptic system while is still has the reconstruction property. As another example of a non-elliptic control system satisfying the assumptions of Theorem \ref{ThmReconstruction}, consider in $\R^3$, with coordinates $(x,y,z)$, the following three vector fields
$$
V_1(x,y,z) = yz\partial_x, \quad V_2(x,y,z) = xz\partial_y, \quad V_3(x,y,z) = xy\partial_z.
$$
Then
$$
\big[V_1,V_2\big] = z^2(y\partial_y - x\partial_x),\quad \big[V_1,V_3\big] = y^2(x\partial_x - z\partial_z),\quad \big[V_2,V_3\big] = x^2(z\partial_z - y\partial_y).
$$
Here $m=6$, and it is easily checked that taking two observation points (i.e. $c=2$), such as the points with coordinates $(1,1,1)$ and $(1,2,3)$, the reconstruction matrix has rank $6$. \vspace{0.2cm}

      \item 
        
        Hypoellipticity (or ellipticity) is also not sufficient for Theorem \ref{ThmReconstruction} to hold.
        Indeed consider
        %More examples of a geometric nature can be investigated, as for instance the rough differential equation associated with the vector fields
      $$
      X_i = \partial_{x_i}+2y_i\,\partial_t, \quad Y_i = \partial_{y_i}-2x_i\,\partial_t,      
      $$
in $\R^{2d+1}$ with coordinates $x\in\R^d,\,y\in\R^d$ and $t\in\R$, used in a sub-Riemannian setting to define the Kohn Laplacian $\sum_{i=1}^d\big(X_i^2+Y_i^2\big)$. They satisfy 
$$
\big[X_i,X_j\big] = 0, \quad \big[Y_i,Y_j\big] = 0, \quad \big[X_i,Y_j\big] = -4\delta_i^j\,\partial_t,
$$
so the reconstruction matrix is always degenerate.

\end{enumerate}

\medskip

Our method of proof is best illustrated with the example of the rolling ball -- see \cite{BrockettDai} for a thorough treatment and
\cite{LyonsQian} for its introduction in a rough path setting. This equation describes the motion of a ball with unit radius rolled on a table without slipping. The position of the ball at time $t$ is determined by the orthogonal projection $x_t\in\R^2$ of the center of the ball on the table (i.e. the point touching the table, with the latter identified with $\R^2$), and by a $(3\times 3)$ orthonormal matrix $\Phi_t\in \OO(\R^3)$ giving the orientation of the ball. Set 
$$
A_1 = \begin{pmatrix}     
            0 & 0 & 1 \\
            0 & 0 & 0 \\
            -1 & 0 & 0  
           \end{pmatrix}
\quad \textrm{and }\quad
A_2 = \begin{pmatrix}     
            0 & 1 & 0 \\
            -1 & 0 & 0 \\
            0 & 0 & 0  
           \end{pmatrix}.
$$
We define right invariant vector fields $V_1,V_2$ on $\OO(\R^3)$ by the formula
$$
V_1(M) = A_1M, \quad V_2(M)=A_2M,
$$
for any $M\in \OO(\R^3)$. The non-slipping assumption on the motion of the ball relates the evolution of the path $x_\bullet$ to that of $\Phi_\bullet$, when the path $x_\bullet$ is $\mcC^1$, as follows
\begin{equation}
\label{EqRollingBallC1}
d\Phi_t = V_1(\Phi_t)\,dx^1_t + V_2(\Phi_t)\,dx^2_t.
\end{equation}
This equation makes perfect sense when $x_\bullet$ is replaced by a rough path $\bfX$ and the equation is understood in a rough path sense. Set $\mathsf{V}=(V_1,V_2)$. Working with invariant vector fields, the solution flow to the rough differential equation
 \begin{equation}
\label{EqRollingBallRDE}
d\varphi_t = \mathsf{V}(\varphi_t)\,{\bfX}(dt)
\end{equation}
is given by the map
$$
\varphi_\bullet : g\in \OO(\R^3) \mapsto \Phi^0_\bullet\,g,
$$
where $\Phi^0_\bullet$ is the solution path to the rough differential equation \eqref{EqRollingBallRDE} started from the identity. We know from the work of Strichartz \cite{Strichartz} on the Baker-Campbell-Dynkin-Hausdorff formula that the solution to the time-inhomogeneous ordinary differential equation \eqref{EqRollingBallC1} is formally given by the time-$1$ map of a {\it time-homogeneous} ordinary differential equation involving a vector field explicitly computable in terms of $V_1,V_2$ and  their brackets, and the iterated integrals of the signal $x_\bullet$, under the form of an infinite series. Truncating this series provides an approximate solution whose accuracy can be quantified precisely under some mild conditions on the driving vector fields. This picture makes perfect sense in the rough path setting of equation \eqref{EqRollingBallRDE} and forms the basis of the flow method put forward in \cite{BailleulFlows}. In the present setting, given a 2-dimensional rough path $\bfX$, with L\'evy area process $\bbA_\bullet$,
and given $0\leq s\leq t\leq 1$, denote by $\psi_{ts}$ the time-$1$ value of the solution path to the ordinary differential equation 
$$
  dz_u = X^1_{ts}V_1(z_u) + X^2_{ts}V_2(z_u) + \bbA_{ts} \big[V_1,V_2\big](z_u), \quad 0\leq u\leq 1,
$$
in $\OO(\R^3)$ started from the identity; that is
$$
\psi_{ts} = \exp\Big(X^1_{ts}A_1 + X^2_{ts}A_2 + \bbA_{ts} \big[A_2,A_1\big]\Big).
$$
Write $\varphi_{ts}$ for $\varphi_t\varphi_s^{-1}$. Then it follows from the results in \cite{BailleulFlows} that there exists some positive constant $c_1$ such that the inequality
\begin{equation}
\label{EqRollingBallApproxFlow}
\big\|\varphi_{ts}-\psi_{ts}\big\|_\infty \leq c_1|t-s|^{3 \alpha}
\end{equation}
holds for all $0\leq s\leq t\leq 1$. Since the vectors $A_1,A_2,\big[A_2,A_1\big]$ form a basis of the vector space $\mcA_3$ of anti-symmetric $(3\times 3)$ matrices, and the exponential map is a local diffeomorphism between a neighbourhood of $0$ in $\mcA_3$ and a neighbourhood of the identity in $\OO(\R^3)$, we get back the coefficient $X^1_{ts},X^2_{ts}$ and $\bbA_{ts}$ from the knowledge of $\varphi_{ts}$ and  relation \eqref{EqRollingBallApproxFlow}, up to an accuracy of order $|t-s|^{3 \alpha}$. This shows that one can reconstruct $\bfX$ from $\varphi_\bullet$, in the sense of Definition \ref{DefnReconstruction}, as the diagonal terms of $\bbX_{ts}$ are given in terms of $X_{ts}$ (see Remark \ref{rem:lieRoughPath}).
One could argue that perfect knowledge of $\varphi_{ts}$ may seem unrealistic from a practical point of view. Note that the above proof makes it clear that it is sufficient to know $\varphi_{ts}$ up to an accuracy of order $|t-s|^{3 \alpha}$ to get the reconstruction result.

\smallskip

%--------------------------------------%
\section{Proofs of the reconstruction theorem}
\label{SectionProofs}
%--------------------------------------%

%%------------------------------------------------------%%
\subsection{Proof I}
\label{SubsectionProofThm}
%%------------------------------------------------------%%

The first proof we give is based on the basic approximation method put forward in \cite{BailleulFlows} to construct the solution flow to a rough differential equation, and used independently later in \cite{BoutaibGyurkoLyonsYang} and \cite{LyonsYang}. As in the rolling ball example, it rests on the fact that one can obtain a good approximation of the solution flow $\varphi_{ts}$ to the rough differential equation \eqref{EqRDE} by looking at the time-$1$ map of an auxiliary time-homogeneous ordinary differential equation constructed from the vector fields $V_i$, their brackets and ${\bfX}_{ts}$. More specifically, let $\psi_{ts}$ stand for the time 1 map of the ordinary differential equation
\begin{equation}
\label{EqODEApprox}
dz_u = \sum_{i=1}^\ell X^i_{ts}V_i(z_u) + \sum_{1\leq j<k\leq \ell}\bbA^{jk}_{ts}\big[V_j,V_k\big](z_u), \quad 0\leq u\leq 1,
\end{equation}
that associates to any $x\in\R^d$ the value at time 1 of the solution to the above equation started from $x$. Then, by Theorem \ref{ThmLyonsUniversal}, there exists a positive constant $c_1$ such that one has
\begin{equation}
\label{EqRDEApprox}
\big\|\varphi_{ts}-\psi_{ts}\big\|_\infty \leq c_1|t-s|^{3 \alpha},
\end{equation}
for all $0\leq s\leq t\leq 1$. The constant $c_1$ depends only on $||V||_{\Lip^3}$ and any upper bound $M$ on the rough path norm of $\bfX$. We write formally
$$
\psi_{ts} = \exp\Big(X^i_{ts}V_i + \bbA^{j<k}_{ts}\big[V_j,V_k\big]\Big),
$$
and set $m=\frac{\ell(\ell+1)}{2}$. Working with $s$ and $t$ close to each other, we expect the coefficients of ${\bfX}_{ts}$ appearing in equation \eqref{EqODEApprox} to lie in any a priori given compact neighbourhood $\mcU$ of $0$ in $\R^m$. The simplest idea to get them back from the knowledge of $\varphi_{ts}$ is then to try and minimize over $\mcU$ the quantity
\begin{align}
  \label{eq:minimizeFlow}
  \Big\|\varphi_{ts}-\exp\Big(A^iV_i + B^{j<k}\big[V_j,V_k\big]\Big)\Big\| _\infty.
\end{align}

\medskip

\begin{remark}
  \label{rem:dossSussmann}
  Note that the ``approximation scheme'' $\psi_{ts}$ is equal to $\varphi_{ts}$ itself in the very special case where $d=1$ and $\ell=1$
  (Actually only $\ell=1$ is necessary ..), by the well-known Doss-Sussmann representation. So if in that case there is a point $y \in \R$ with $V_1(y)\neq 0$, the map $\frak{f} : a\in\R\mapsto \exp\big(aV_1\big)(y)$, is a local diffeomorphism between a neighbourhood of $0$ in $\R$ and a neighbourhood $\mcV$ of $y$ in $\R$. One thus has $X_{ts}=\frak{f}^{-1}\big(\varphi_{ts}(y)\big)$, for $s$ and $t$ close enough for $\varphi_{ts}(y)$ to be in $\mcV$. The reconstruction is perfect in that case. (Note that a $1$-dimensional rough path does not have an ``area''.)
\end{remark}

\medskip

\begin{proof}[Proof of Theorem \ref{ThmReconstruction}]

\textbf{ Sufficiency. }
Assume for the moment $d \ge m$, and suppose that at some point $y\in\R^\dode$ the vectors 
$$
\Big(V^i(y),\big[V_j,V_k\big](y)\,;\,1\leq i\leq \ell, \,1\leq j < k \leq \ell\Big),
$$ 
are independent. Define a map $\Psi_y$ from $\R^m$ to $\R^d$ setting
$$
\Psi_y(A,B) = \exp\Big(A^iV_i + B^{j<k}\big[V_j,V_k\big]\Big)(y).
$$
By Lemma \ref{lem:flowDiffeo} there exists two explicit positive constants $\epsilon_1, \epsilon_2$, depending only on the $\Lip^3$-norm of $V_1,\dots,V_\ell$, such that for for any two points ${\bf a},{\bf a}'$ in the ball $\mcU := B_{\epsilon_1}(0)$ of $\R^m$, we have
\begin{align}
  \label{eq:lowerBound}
  \big\|\Psi_y({\bf a})-\Psi_y({\bf a}')\big\| \geq \epsilon_2 \|{\bf a}-{\bf a}'\|.
\end{align}

We claim that any minimizer $({\sf A}, {\sf B})$ in $\mcU$ of the expression
\begin{equation}
\label{EqApproachFlow}
\Big|\varphi_{ts}(y)-\exp\Big(A^iV_i + B^{j<k}\big[V_j,V_k\big]\Big)(y)\Big|
\end{equation}
satisfies the identity
\begin{equation}
\label{EqAlmostRP}
({\sf A}, {\sf B}) - \big(X_{ts},\bbX_{ts}\big) = O\Big(|t-s|^{3 \alpha}\Big),
\end{equation}
for $t-s$ small enough, with a constant in the $O(\cdot)$ term independent of the minimizer. Assume, by contradiction, the existence for every $M>0$ and $\delta>0$, of times $(s,t)$ with $0\leq t-s\leq \delta$, and some minimizer $\big({\sf A}_{ts}, {\sf B}_{ts}\big)$ in $\mcU$ such that 
$$
\Big|\big({\sf A}_{ts}, {\sf B}_{ts}\big) - \big(X_{ts},\bbX_{ts}\big)\Big| \geq M |t-s|^{3 \alpha}.
$$
Then the inequality
$$
\Big|\Psi_y\big({\sf A}_{ts}, {\sf B}_{ts}\big) - \Psi_y \big(X_{ts},\bbX_{ts}\big)\Big| \geq \epsilon_2 M\,|t-s|^{3 \alpha}
$$
would follow from \eqref{eq:lowerBound}, giving, for a choice of $M =\frac{2c_1}{\epsilon_2}$, the conclusion
$$
\Big|\Psi_y\big({\sf A}_{ts}, {\sf B}_{ts}\big) - \varphi_{ts}(y)\Big| > c_1\,|t-s|^a,
$$
contradicting identity \eqref{EqRDEApprox}, where $\big(X_{ts},\bbX_{ts}\big)$ belongs to $\mcU$ for $\delta$ small enough, and the fact that $\big({\sf A}_{ts}, {\sf B}_{ts}\big)$ is a minimizer. This proves Theorem \ref{ThmReconstruction} in the special case where $d \geq m$ and where for some $y\in\R^\dode$ the family $\big(V^i(y),\big[V_j,V_k\big](y)\,;\,1\leq i\leq \ell, \,1\leq j < k \leq \ell\big)$ is free.

\medskip

To handle the general case, identify $\big(\R^d\big)^c$ and $\R^{cd}$, and denote by $z=(y_1,\dots,y_c)$ a generic element of $\R^{cd}$, with $y_i\in\R^d$. Introduce the vector fields ${\sf V}_i$ on $\R^{cd}$, given by the formula
$$
W_i(z) = \begin{pmatrix}     
            V_i(y_1) \\
            \vdots \\
            V_i(y_c)
           \end{pmatrix}.
$$
These vector fields satisfy, under the assumptions of Theorem \ref{ThmReconstruction}, the restricted assumptions under which we have proved Theorem  \ref{ThmReconstruction} above. So this special case applies and implies the general case. The above proof shows in particular that 
\begin{equation*}
\Big|\big({\sf A}_{ts}, {\sf B}_{ts}\big) - \big(X_{ts},\bbX_{ts}\big)\Big| \leq \frac{2c_1}{\epsilon_2}\,|t-s|^{3 \alpha}. 
\end{equation*}

\textbf{ Necessity. }

  Define
  \begin{align*}
    Z := \spann\left\{ V_1^i(y), \dots, V_\drp^i(y), [V_1, V_2]^i(y), \dots, [V_{\drp-1}, V_\drp]^i(y) : i=1,\dots,\dode, y \in \R^\dode, i=1,\dots,\dode \right\}.
  \end{align*}
  If the assumption of the theorem is not satisfied then $Z \not= \R^m$.
  Hence pick $v \in Z^T$, $v \not= 0$.
  Then for every $\mathbf{a} \in \R^m$, every $y \in \R^\dode$
  \begin{align*}
    \Psi_y( \mathbf{a} )
    =
    \Psi_y( \mathbf{a} + v ).
  \end{align*}
  Hence the null rough path $(X,\mathbb{X}) := (0,0)$ and the rough path
  \begin{align*}
    (\bar X_{ts}, \mathbb{\bar X}_{ts})
    :=
    \left( 
    \left( \begin{array}{c}
              v^1 \\
              \dots \\
              v^\ell
            \end{array} \right) (t-s), 
    \left(
    \begin{matrix}
      0                   &         v^{\ell+1} & \dots &         v^{\ell+d-1} \\
      -        v^{\ell+1} & 0                  & \dots & \dots \\
      \dots               & \dots              & \dots & v^m \\
      -      v^{\ell+d-1} & \dots              & -v^m & 0
    \end{matrix}
    \right)
    (t-s)
    \right),
  \end{align*}
  have the same effect on the rough differential equation.
  Hence reconstruction is not possible.
\end{proof}

\begin{remarks}
  \label{rem:afterMainTheorem}

\noindent {\bf 1.} {\it The above proof shows that any minimizer to problem \eqref{EqApproachFlow} satisfies identity \eqref{EqAlmostRP} if $0\leq t-s \leq \delta$, provided $\delta>0$ is chosen such that $\big(X_{ts},\bbX_{ts}\big)\in\mcU$. The results of \cite{BailleulFlows} show that $\delta$ is of order $\big(1+\|{\bfX}\|_\alpha\big)^{-3}$; this quantity is a priori unknown since $\bfX$ itself is unknown. In practice, one should work with a sufficiently small a priori given $\delta$ and refine it if necessary.}

\medskip

\noindent {\bf 2.} {\it Note that the use of ordinary differential equations as a tool makes the above method perfectly suited for dealing with rough differential equations \eqref{EqRDE} with values in manifolds. This would not have been the case if we had replaced the exponential map used to define $\psi_{t}$ by a Taylor polynomial (as we shall
do below, in the second proof of Theorem \ref{ThmReconstruction}), which does not have any intrinsic meaning on a manifold. Denote by $Vf$ the derivative of a function $f$ in the direction of a vector field $V$. In the manifold setting, the reconstructability condition takes the following form. There exists a (smooth) function on the manifold and some points $x_1,\dots,x_c$ such that the following matrix has rank $m$,
\begin{align*}
    \begin{pmatrix}
      \big(V_1f\big)(x_1) & \cdots & \big(V_\ell f\big)(x_1) & \big([V_1,V_2]f\big)(x_1) & \cdots &  \big([V_{\ell-1},V_\ell]f\big)(x_1)  \\
      \vdots      &              & \vdots          & \vdots                  &            &  \vdots   \\
      \big(V_1f\big)(x_c) & \cdots & \big(V_\ell f\big)(x_c) & \big([V_1,V_2]f\big)(x_c) & \cdots &  \big([V_{\ell-1},V_\ell]f\big)(x_c)  
    \end{pmatrix}.
\quad
  \end{align*}
With an eye back on the rolling ball example, it suffices in that case to observe the flow at only one point and to take as function $f$ the logarithm map, from $\OO(\R^3)$ to the linear space of antisymmetric $(3\times 3)$ matrices.}
\end{remarks}

%%----------------------------------------------%%
\subsection{Proof II}
\label{SubsectionReconstructionAlgorithm}
%%----------------------------------------------%%

The proof of Theorem \ref{ThmReconstruction} relied on the flow approximation
to rough differential equations. In this paragraph we show that the same result
can be obtained using an Euler-type approximation, which leads
to a computationally less expensive solution.

\begin{proof}[Second proof of Theorem \ref{ThmReconstruction}]
  \textbf{ Sufficiency. }

  Assume $\dode \ge m$ and the existence of a point $y\in\R^\dode$ where the vectors $\big(V^i(y),\big[V_j,V_k\big](y)\,;\,1\leq i\leq \ell, \,1\leq j < k \leq \ell\big)$ form a free family. Instead of approximating the flow of the rough differential equation by the time 1 map $\Psi_y$ we use the Taylor approximation
  \begin{align*}
    \Phi_y( A, B ) := y + A^i V_i(y) + \Big(B^{j < k} \left[ V_j, V_k \right] + \frac{1}{2} A^i A^j V_i V_j\Big)(y).
  \end{align*}
  By Lemma \ref{lem:taylorDiffeo}, there exist some positive constants $\epsilon_1, \epsilon_2$ such that we have 
  \begin{align*}
    \Big|\Phi_y({\bf x})-\Phi_y({\bf x}')\Big| \geq \epsilon_2 |{\bf x}-{\bf x}'|,
  \end{align*}
  for any pair of points ${\bf x},{\bf x}'$ in the ball $\mcU := B_{\epsilon_1}( 0 )$ of $\R^m$. The proof then follows the exact same steps as above.
\end{proof}

\subsection{The reconstruction algorithm}

Based on Proof II, each step of the reconstruction scheme can be described in the following simple terms.

\smallskip

{\sf
  1. Observe the solution increments $\phi_{ts}(x_1), \dots, \phi_{ts}(x_c)$
      started from the points $x_1, \dots x_c$ as given in the statement of Theorem 2.

  2. Minimize the quadratic target function
  \begin{align*}
    \sup_{\ell=1,\dots,c} \Big| \phi_{t,s}(x_\ell) - \Big\{A^i V_i(x_\ell) + \Big(B^{j < k} \left[ V_j, V_k \right] + \frac{1}{2} A^i A^j V_i V_j\Big)(x_\ell)\Big\}\Big|.
  \end{align*}
  The minimizer $\mathsf{A},\mathsf{B}$ is an approximation for $\mathcal{R}_{s,t}$ (defined in Remark \ref{rem:lieRoughPath}).
}

\section{Examples}
\label{SectionExamples}
We present some examples of dynamical systems where the conditions of Theorem \ref{ThmReconstruction} hold.
The examples are physical, in the sense that they model dynamics that can be realized as concrete machines.

\subsection{Rolling ball and its implementations}
The rolling ball was already considered at the end of Section \ref{SectionMainResult}.

Another incarnation is available in the field of quantum control.
In fact the problem studied in \cite{bib:boscainEtAl2205}
is of the same form (see equation (13) ibid).
We should point out a caveat in this case: due to quantum effect the dynamics of the system
are only well-described by the differential equation of the rolling ball for forcing signals with derivative of small amplitude
.\footnote{Personal communication with Ugo Boscain.}
This is of course not the case for a general rough path.

\subsection{Continuously variable transmission}
\begin{figure}[h!]
  \centering
  \includegraphics[width=0.45\textwidth]{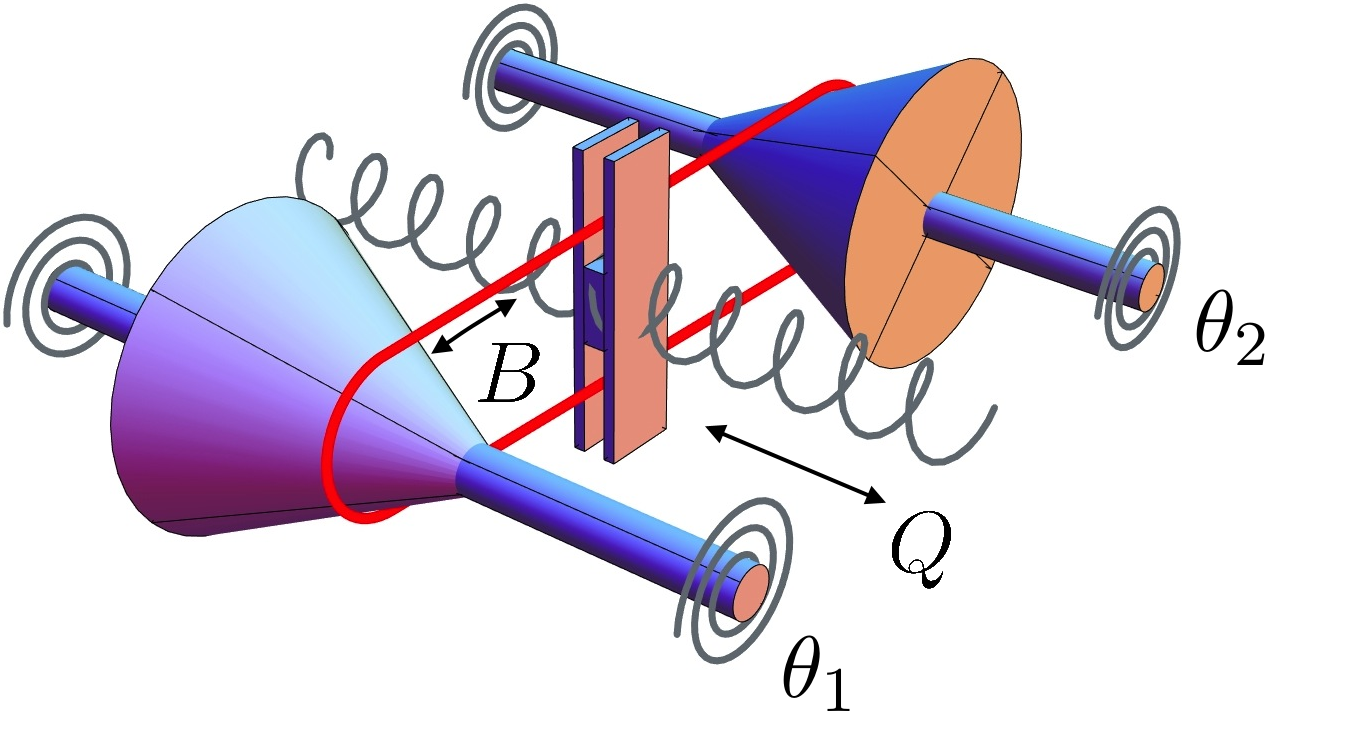}\\
  \caption{
    An implementation of
  the continuously variable transmission. The image is from \cite{ModinVerdier}, with kind permission of the authors.
    }
  \label{fig:cvt}
\end{figure}
The uncontrolled dynamics of the system illustrated in Figure \ref{fig:cvt}
is studied in \cite{ModinVerdier}.
The controlled setting is as follows.
The statespace is given by $\theta_1, \theta_2$ the rotation angles of the two cones,
$B$ the vectical position %(in $x$ direction)
of the belt and $Q$ the translation of the belt, constrained to be in an interval, say $(0,1)$.
The controller acts directly on $B$ and $Q$ and so for
$Q(t) \in (0,1)$ the corresponding equations are given by
\begin{align*}
  d\theta_1(t) &= \frac{1}{Q(t)} dB(t), \\
  d\theta_2(t) &= \frac{1}{1-Q(t)} dB(t), \\
  dB(t) &= dX^1(t), \\
  dQ(t) &= dX^2(t).
\end{align*}
With $x := (\theta_1, \theta_2, B, Q)$ this reads as
\begin{align*}
  dx_t = V_i(x_t) dX^i_t,
\end{align*}
where
\begin{align*}
  V_1(x) = 
  \left(
  \begin{matrix}
    \frac{1}{q} \\
    \frac{1}{1-q} \\
    1 \\
    0
  \end{matrix}
  \right),
  \quad
  V_2(x) = 
  \left(
  \begin{matrix}
    0 \\
    0 \\
    0 \\
    1.
  \end{matrix}
  \right),
\end{align*}
We calculate
\begin{align*}
  \left[ V_1, V_2 \right](x)
  =  
  \left(
  \begin{matrix}
    \frac{1}{q^2} \\
    -\frac{1}{(1-q)^2} \\
    0 \\
    0
  \end{matrix}
  \right).
\end{align*}
The condition of Theorem \ref{ThmReconstruction} is hence satisfied with $c=1$ and any point
$x = (\theta_1,\theta_2,q,b)$ with $q \in (0,1)$.

\subsection{Unicycle}
\begin{figure}[h!]
  \centering
  \includegraphics[width=0.25\textwidth]{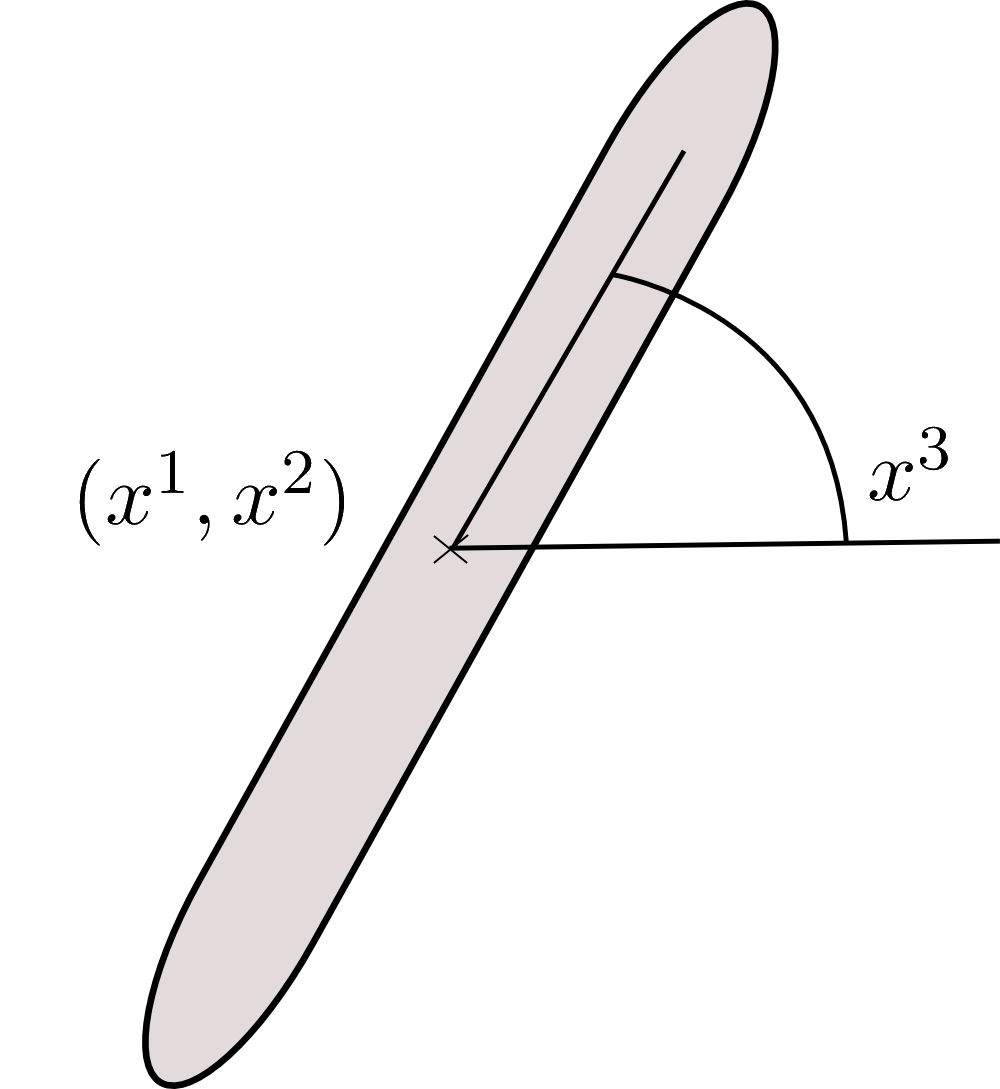}\\
  \caption{
    The unicycle.
    }
  \label{fig:unicycle}
\end{figure}
Consider a unicycle on the sphere with position $(x^1,x^2)$ and orientation $x^3$ sketched in Figure \ref{fig:unicycle}.
The controller acts by changing the orientation $x^3$ and by moving the unicycle foward or backward along the current orientation.
The dynamics hence reads as (see Example 4.3.2 in \cite{bibBloch})
\begin{align*}
  dx_t = V_1(x_t) dX^1_t + V_2(x_t) dX^2_t,
\end{align*}
where
\begin{align*}
  V_1(x) :=
  \left(
  \begin{matrix}
    \cos( x^3 ) \\
    \sin( x^3 ) \\
    0
  \end{matrix}
  \right),
  \qquad
  V_2(x) :=
  \left(
  \begin{matrix}
    0 \\
    0 \\
    1
  \end{matrix}
  \right).
\end{align*}

Then
\begin{align*}
  \left[ V_1, V_2 \right](x)
  =
  \left(
  \begin{matrix}
    \sin(x^3) \\
    -\cos(x^3) \\
    0
  \end{matrix}
  \right).
\end{align*}
The condition of Theorem \ref{ThmReconstruction} is hence satisfied with $c=1$ and any point.

%-------------------------%
\section{Applications}
\label{SectionApplications}
%-------------------------%

We describe in this section two applications of the reconstruction Theorem \ref{ThmReconstruction}.

%%--------------------------------------------------------------%%
\subsection{Filtering and maximum likelihood estimator}
\label{SubsectionFilteringAndMLE}
%%--------------------------------------------------------------%%

We give a brief overview on two recent results in the area of stochastic filtering and maximum likelihood estimation which both emphasize the need in different practical situations to measure signals in a rough path sense. The point of the present work is that if these real life signals can be used as input to an \emph{additional} physical system that is modelled by a rough differential equation satisfying the assumptions of our reconstruction theorem, then one can indeed have a good approximation of the rough signal, which suffices for practical purposes.

\medskip

{\bf a) Filtering.} Consider the following multi-dimensional two-component stochastic differential equation
\begin{align*}
  x_t &= x_0 + \int_0^t V(x_r,y_r) dr + \int_0^t V_j(x_r,y_r)\,{\circ dB^j_r} \\
     &\qquad\;\;\, + \int_0^t V'_k(x_r,y_r)\,{\circ dy^k_r}, \\
  y_t &= \int_0^t h(x_r, y_r)\,dr + W_t,
\end{align*}
where the path $y_\bullet$ is observed and the path $x_\bullet$ is unobserved. The letters $W$ and $B$ stand here for independent Brownian motions. The stochastic filtering problem consists in calculating the best guess (in $L^2$ sense) of the signal given the observation 
$$
\pi_t = \E\big[ f(x_t)\,\big|\,\sigma( y_s : s \le t ) \big].
$$
Here $f$ is some smooth nice enough function. (More generally one is interested in the \emph{distribution} of $x_t$ given the past of $y_\bullet$.) From a practitioner's point of view it is highly desirable that the estimation procedure be continuous in the observation path \cite{bibClark}. If the vector fields $V'_k$ are null (the uncorrelated case) or if $y_\bullet$ is 1-dimensional, it has been shown that this in fact the case: $\pi_t$ is continuous with respect to the path $y_{[0,t]}$, where distance is measured in supremum norm (see \cite{bibClark,bibDavis}). Counterexamples show that in the general case this is \emph{not} true. As recalled above in section \ref{SectionRoughPaths}, Stratonovich integrals with continuous semimartingale integrand and the corresponding rough paths integrals against the rough path lift of the integrand coincide whenever both make sense. So one can recast the above equation for the path $x_\bullet$ as a rough differential equation involving the rough path lift $\bf y_\bullet$ of the observed path $y_\bullet$. 

\begin{theorem}[Theorems 6 and 7 in \cite{CrisanDiehlOberhauser}]
Under appropriate assumptions on the vector fields\footnote{This refers to boundedness and sufficient smoothness; no bracket assumption as in Theorem \ref{ThmReconstruction} is needed for this result.}, there exists  a {\emph continuous} deterministic function $\frak f$ on the rough path space such that
  \begin{align*}
    \pi_t = {\frak f}\big(\mathbf{y}_{[0,t]}\big),
  \end{align*}
where $\mathbf{y}_\bullet$ is the Stratonovich lift of $y_\bullet$ to a rough path.
\end{theorem}

  Note that even if we had an observation model
  \begin{align*}
    dy = h(x) dt + \sigma(y) dW,
  \end{align*}
  with $\sigma$ satisfying the assumptions of Theorem \ref{ThmReconstruction} in the ideal case where $c = 1$, (for which we need $d_Y \ge d_W (d_W + 1) / 2$), observing just Y is \emph{still not enough}. This comes from the fact that the (random!) drift term introduces an error of order $|t-s|^1$. It is therefore indeed necessary to have an additional ``measuring device'', which is fed the observation $y$ and is modelled as an RDE satisfying the assumptions of Theorem \ref{ThmReconstruction}.

\bigskip

{\bf b) Statistics.} Consider now the problem of estimating the parameter $A \in \textrm{L}(\R^d)$ in the stochastic differential equation
\begin{align*}
  dy_t=A\,h\left(y_t \right) dt + \Sigma \left(y_t\right) dB_t
\end{align*}
driven by a Brownian motion $B$ or other Gaussian process, by observing $y_\bullet$ on some fixed time interval. Under appropriate conditions, the measures on path space for different $A$ are mutually absolutely continuous, so one can use the method of maximum likelihood estimation, leading to an estimator $\widehat{A}_T$ for $A$. Denote by $\mathbf{y}_\bullet$ the Stratonovich lift of $y_\bullet$ into a rough path.

\begin{theorem}[Theorem 2 in \cite{bibDiehlFrizMai}]
  Under appropriate assumptions on the functions $h$ and $\Sigma$, there exists a subset $\mathcal{D}$ of the space of weak geometric $\alpha$-H\"older rough paths and a {\emph continuous} deterministic function $\mathcal{A}_T: \mathcal{D} \to \textrm{L}(\R^d)$ such that $\mathbf{y}_\bullet$ belongs almost-surely to $\mathcal{D}$ and $\mathcal{A}_T({\mathbf{y}}_\bullet)$ is almost-surely equal to $\widehat{A}_T$.
\end{theorem}

Again, there exist counterexamples that show that taking only the \emph{path} of the observation as an input does \emph{not} yield a robust estimation procedure. 

\medskip

In both examples the practitioner is hence left with the task of actually \emph{recording} the rough path $\mathbf{Y}$ in order to implement on a practical basis the above theoretical results. The reconstruction theorem provides conditions under which this can be done. This requires from the practitioner to observe another system where the signal of interest serves as an input.

%---------------------%
\section{Appendix}
\label{SectionAppendix}
%---------------------%

We gather in this appendix a few elementary lemmas that were used in the proof of the reconstruction theorem. We start with a basic result that is used in the two lemmas below.

\begin{lemma}
  \label{lem:appendixBoundOnD2}
Let $W : \R^p \times \R^\dode \to \R^\dode$ be a function of class $C^2$ such that the vector field $W(a,\cdot)$ is Lipschitz continuous for any fixed $a \in \R^p$. Let $\phi_{t}(a,\cdot)$ be the flow at time $t$ started at time $0$ to the ordinary differential equation
  \begin{align*}
    dx_t = W(a,x_t) dt.
  \end{align*}
  Then
  \begin{align*}
    c_0 := \sup_{|a| \le 1, x \in \R^\drp}\big| D_a^2 \phi_{t}(a,x)\big| < \infty.
  \end{align*}
\end{lemma}

\begin{Dem}
  Note that $\phi_t$ is the projection on the last $\drp$ coordinates of the flow $\psi_t$ to the enlarged equation
  \begin{align*}
    da_t &= 0 \\ 
    dx_t &= W(a_t,x_t) dt.
  \end{align*}
  Hence, for $k=1, \dots \dode$, $i=1,\dots,p$, 
  \begin{align*}
    d\partial_{a_i} \phi_t^k
    &=
    \partial_{y_h} W^k(a_t,x_t) \partial_{a_i} \phi^h_t dt
    +
    \partial_{a_r} W^k(a_t,x_t) dt \\
    d\partial_{a_j, a_i} \phi_t^k
    &=
    \partial_{y_g, y_h} W^k(a_t,x_t) \partial_{a_j} \phi^g_t \partial_{a_i} \phi^h_t dt
    +
    \partial_{a_s, y_h} W^k(a_t,x_t) \partial_{a_i} \phi^h_t dt
    +
    \partial_{y_h} W^k(a_t,x_t) \partial_{a_j, a_i} \phi^h_t dt \\
    &\qquad
    +
    \partial_{y_g, a_r} W^k(a_t,x_t) \partial_{a_j} \phi^g_t dt
    +
    \partial_{a_s, a_r} W^k(a_t,x_t) dt; 
    \end{align*}
    see for example \cite[Chapter 4]{bibFrizVictoir}. An application of Gr\"onwall's lemma now gives the desired result.
\end{Dem}

\begin{lemma}[Non-degeneracy of the flow approximation]
  \label{lem:flowDiffeo}
  Let $N \ge n$ and $W_1, \dots, W_n \in \Lip^{3}$  be vector fields that are linearly independent at some point some $y \in \R^N$. Then the function
  \begin{align*}
         \Psi_y: \R^n &\to \R^N \\
    {\bfa} = (a_1,\dots,a_n) &\mapsto \exp\big(a_1 W_1 + \dots + a_n W_n\big)(y),
  \end{align*}
  is $C^2$ and there is a neighbourhood $\mcU$ of $0$ in $\R^n$ and a positive constant $\epsilon_1$ such that for ${\bfa},{\bfa}'\in\mcU$ we have
  \begin{align*}
    \big|\Psi_y({\bfa}) - \Psi_y({\bfa}')\big| \ge \epsilon_1 |{\bfa} - {\bfa}'|.
  \end{align*}
\end{lemma}

The proof below makes it clear that one can choose
  \begin{align*}
    \epsilon_1 &:= \Big\|\left (D\Psi_y(0)^{T}D\Psi_y(0)\right)^{-1} D\Psi_y(0)^{T}\Big\|^{-1} \\
    \epsilon_2 &:= \frac{1}{2 c_0},
  \end{align*}
with the constant $c_0$ of Lemma \ref{lem:appendixBoundOnD2}.

\medskip

\begin{Dem}
  From classical results by Gr\"onwall, see for example \cite[Theorem 14.1]{bibHairerEtAl1991},  we know that $\Psi_y \in C^2$ (see also Lemma \ref{lem:appendixBoundOnD2}). By Taylor's theorem we have
  \begin{align*}
    &\Bigl| \exp\big( a_1 W_1 + \cdots + a_n W_n \big)(y) - \big( y + a_1 W_1 + \cdots + a_n W_n \big)(y) \Bigr| \\
    & \qquad \leq C ||W||_{\Lip^3} \left( |a_1|^2 + \dots + |a_n|^2 \right),
  \end{align*}
for some positive constant $C$ that depends only on $n$. Hence
  \begin{align*}
    D_0\Psi_y = \left( W_1 \cdots W_n \right) \in \R^{N \times n},
  \end{align*}
has rank $n$, by assumption. Now
  \begin{align*}
    \Psi_y({\bfa}) - \Psi_y({\bfa}') = D_0\Psi_y({\bfa} - {\bfa}') + \big( D\Psi_y({\bfa}') - D\Psi_y(0) \big) \big({\bfa} - {\bfa}'\big) + O\big(|{\bfa}-{\bfa}'|^2\big),
  \end{align*}
  where the last term is of the form 
  $$
  |a-\bar a|^2\,\max_{z \in U}\big|D^2 \Psi_y(z)\big|,
  $$ 
  for ${\bfa},{\bfa}' \in U$. Now, since $D_0\Phi_y$ has rank $n$, we have
  $$
  \big|D\Psi_y(0) (a-\bar a)\big| \ge 2 \epsilon_1 |a-\bar a|
  $$
 with $\epsilon_1 := \Big\|\left (D_0\Psi_y^{T}D_0\Psi_y\right)^{-1} D_0\Psi_y^T\Big\|^{-1} > 0$. Note that by Lemma \ref{lem:appendixBoundOnD2}, $\|D^2_{\bfa}\Phi_y\| \leq c_0$. Then choosing $\mcU := B_{\epsilon_2}(0)$, with
  \begin{align*}
    \epsilon_2 = \frac{1}{ 2 c_0},
  \end{align*}
  the second and third terms are dominated by $\frac{\epsilon_1}{2} |{\bfa} - {\bfa}'|$, which yields the desired result.
\end{Dem}

Last, we provide a version of the previous lemma adapted to the 'numerical scheme' put forward in Section \ref{SubsectionReconstructionAlgorithm}.

\begin{lemma}[Non-degeneracy of Taylor approximation]
  \label{lem:taylorDiffeo}
  Let $V_1, \dots, V_\drp \in \Lip^3(\R^\dode)$. Assume that $\dode \ge m := \drp + \drp (\drp - 1) / 2$ and that moreover at some point $y \in \R^\dode$, the vectors
  \begin{align*}
    V_i(y) &: i = 1, \dots, \drp \\
    \left[ V_i, V_j \right](y) &: i < j
  \end{align*}
  are independent. Then
  \begin{align*}
    \Phi_y: \R^{m} &\to \R^{\dode} \\
    \big(a_1,\dots,a_\drp, a_{1,2}, \dots a_{\drp-1,\drp}\big) &\mapsto
    y
    + \big( a_1 V_1 + \dots + a_\ell V_\ell \big)(y)
    + \sum_{i < j} a_{ij} \left[ V_i, V_j \right](y) \\
    &\qquad
    + \sum_{i,j} a_i a_j \frac{1}{2} V_i V_j(y),
  \end{align*}
  is $C^\infty$.

  Moreover there is a neighborhood $U$ of $0$ and $\epsilon_1 > 0$ such that for $a, \bar a \in U$ we have
  \begin{align*}
    \big|\Phi_y(a) - \Phi_y(\bar a)\big| \ge \epsilon_1 |a - \bar a|.
  \end{align*}
  We can choose $\epsilon_1 = 1/||D\phi_y(0)^{-1}||$ and $U = B_{\epsilon_2}( 0 )$ with
  $\epsilon_2 = \frac{1}{ 2 \sum_{i,j} |V_i V_j(y)| }$.
\end{lemma}
\begin{remark}
The statement is really about a set of some vectors, not about vector fields. Nonetheless we state it in this form, since this is how we need it in the main text. Its proof is almost-identical to the previous one, so we omit it.
\end{remark}


\bigskip
\bigskip



\begin{thebibliography}{10}


\bibitem{M2Course}
Bailleul, I.,
\newblock A flow-based approach to rough differential equations.
\newblock arXiv:1404.0890, 2014.

\bibitem{BailleulFlows}
Bailleul, I.,
\newblock Flows driven by rough paths.
\newblock {\it arXiv:1203.0888}, to appear in Revista Matematica Iberoamericana 2015.

\bibitem{BailleulRegularity}
Bailleul, I.,
\newblock Regularity of the It\^o-Lyons map.
\newblock {\it arXiv:1401.1147}, 2013.

\bibitem{bibBloch}
Bloch, Anthony M.
\newblock Nonholonomic mechanics and control. Vol. 24.
\newblock Springer, 2003.

\bibitem{bibBonfiglioliEtAl}
Bonfiglioli, Andrea, Ermanno Lanconelli, and Francesco Uguzzoni.
\newblock Stratified Lie groups and potential theory for their sub-Laplacians.
\newblock Berlin: Springer, 2007.

\bibitem{bib:boscainEtAl2205}
\newblock Boscain, Ugo, Thomas Chambrion, and Gr\'egoire Charlot.
\newblock "Nonisotropic 3-level quantum systems: complete solutions for minimum time and minimum energy."
\newblock arXiv preprint quant-ph/0409022 (2004).

\bibitem{BoutaibGyurkoLyonsYang}
 Boutaib, Y. and Gyurko, L. and Lyons, T. and Yang, D.,
 \newblock Dimension-free Euler estimates of rough differential equations.
\newblock  arXiv:1307.4708, 2013.

\bibitem{BrockettDai}
Brockett, R. W., and Dai, L.,
\newblock "Non-holonomic kinematics and the role of elliptic functions in constructive controllability."
\newblock {\em Nonholonomic motion planning. Springer US,} 1--21, 1993.

\bibitem{bibClark}
Clark, J. M. C.,
\newblock The design of robust approximations to the stochastic differential equations of nonlinear filtering.
\newblock {\em Communication systems and random process theory}, 25: 721-734, 1978.

\bibitem{Davie}
Davie, A.M.
\newblock Differential equations driven by rough paths: an approach via discrete approximation.
\newblock Appl. Math. Res. Express, 2007 

\bibitem{bibDavis}
Davis, M. H. A.,
\newblock Pathwise nonlinear filtering with correlated noise.
\newblock {\em The Oxford Handbook of Nonlinear Filtering}, 403--424. Oxford Univ. Press, Oxford, 2011.

\bibitem{bibDiehlFrizMai}
Diehl, J. and Friz., P. and Mai, H.
\newblock Pathwise stability of likelihood estimators for diffusions via rough paths
\newblock arXiv preprint arXiv:1311.1061, 2013.

\bibitem{CaruanaLevyLyons}
Caruana, M. and L\'evy, T. and Lyons, T.,
\newblock Differential equations driven by rough paths.
\newblock {\em Lecture Notes in Mathematics},  vol. 1908, 2007.

\bibitem{CassLyonsMeanField}
Cass, T. and Lyons, T.,
\newblock Evolving communities with individual preferences.
\newblock {\it P.L.M.S}, doi:10.1112, 2014.

\bibitem{CrisanDiehlOberhauser}
Crisan, D. and Diehl, J. and Oberhauser, H.,
\newblock Robust filtering: correlated noise and multidimensional observation.
\newblock {\em Annals of Applied Probability}, 23(5):2139--2160, 2013

\bibitem{Fliess}
Fliess, Michel.
\newblock "Fonctionnelles causales non linéaires et indéterminées non commutatives."
\newblock {\em Bulletin de la société mathématique de France } 109 (1981): 3-40.

\bibitem{bibFrizVictoir}
Friz, Peter K., and Nicolas B. Victoir.
\newblock Multidimensional stochastic processes as rough paths: theory and applications.
\newblock {\em Cambridge University Press}, Vol. 120, 2010.

\bibitem{FrizVictoir}
Friz, P. and Victoir, N.,
\newblock Mulitdimensional processes as rough paths.
\newblock {\em Cambridge Studies in Advanced Mathematics},  vol. 120, C.U.P., 2010.

\bibitem{FrizHairer}
Friz, P. and Hairer, M.,
\newblock A Course on Rough Paths.
\newblock Springer, 2014.

\bibitem{Gubinelli}
Gubinelli, Massimiliano.
\newblock "Controlling rough paths."
\newblock Journal of Functional Analysis 216.1 (2004): 86-140.

\bibitem{GubinelliImkellerPerkowski}
Gubinelli, M. and Imkeller, P. and Perkowski, N.,
\newblock Paracontrolled distributions and singular PDEs.
\newblock {\it arXiv:1210.2684}, 2013.

\bibitem{GyurkoLyonsKontkowskiField}
Gyurko, L.and Lyons, T. and Kontkowski, M. and Field, J.,
\newblock Extracting information from the signature of a financial data stream.
\newblock {\it arXiv:1307.7244}, 2013.

\bibitem{HairerRegularity}
Hairer, M.,
\newblock A theory of regularity structures.
\newblock {\em Inventiones Math.}, {\it arXiv:1303.5113}, 2014. 

\bibitem{bibHairerEtAl1991}
Hairer, E. and Nørsett, S. and Wanner, G.,
\newblock Solving Ordinary Differential Equations, Vol. 1: Nonstiff Problems.
\newblock {\em Springer}, 94--98, 1987.

\bibitem{Hirsch}
Hirsch, M.,
\newblock Differential topology.
\newblock {\em Springer}, 1976.

\bibitem{LedouxQianZhang}
Ledoux, M. and Qian, Z. and Zhang, T.,
\newblock Large deviations and support theorem for diffusion processes via rough paths.
\newblock{\em Stochastic Process. Appl.}, 102(2):265--283, 2002.

\bibitem{LevinLyonsNi}
Levin, D. and Lyons, T. and Ni, H.,
\newblock Learning from the past, predicting the statistics for the future, learning an evolving system.
\newblock {\it arXiv:1309.0260}, 2013.

\bibitem{Lyons98}
Lyons, T., 
\newblock Differential equations driven by rough signals.
\newblock {\em Rev. Matematic\'a Iberoamericana}, 14(2):215--310, 1998. 

\bibitem{LyonsStFlour}
Caruana, M. and L\'evy, T. and Lyons, T., 
\newblock Differential equations driven by rough paths.
\newblock {\em Saint Flour lecture notes}, Lecture Notes in Mathematics, 1908, 2007. 

\bibitem{LyonsQian}
Lyons, T. and Qian,Z.,
\newblock System control and rough paths.
\newblock {\em Oxford Univ. Press}, 2002.

\bibitem{LyonsQianFlow}
Lyons, T. and Qian, Z.,
\newblock Flow equations on spaces of rough paths. 
\newblock {\em J. Funct. Anal.}, 149(1):135--159, 1997. 

\bibitem{LyonsVictoir}
Lyons, T. and Victoir, N.,
\newblock Cubature on Wiener space. Stochastic analysis with applications to mathematical finance. 
\newblock {\em Proc. R. Soc. Lond. Ser. A Math. Phys. Eng. Sci.} 460(2041), 169--198, 2004.

\bibitem{LyonsYang}
Lyons, T. and Yang, D.,
\newblock Rough differential equation in Banach space driven by weak geometric p-rough path.
\newblock arXiv:1402.2900, 2014.

\bibitem{ModinVerdier}
Modin, Klas, and Olivier Verdier. 
\newblock Integrability of Nonholonomically Coupled Oscillators.
\newblock Discrete and Continuous Dynamical Systems. Series A 34.3 (2014): 1121-1130.

\bibitem{Strichartz}
Strichartz, R.,
\newblock The Campbell-Baker-Hausdorff-Dynkin formula and solutions of differential equations.
\newblock {\em J. Funct. Anal.}, 72(2):320--345, 1987.

\bibitem{Young}
Young, L.C.,
\newblock An inequality of the {H}\"older type, connected with {S}tieltjes integration.
\newblock{\em Acta Math.}, 67:251--282, 1936.

\bibitem{Zeidler}
Zeidler, E.,
\newblock Nonlinear Functional Analysis and Its Applications: Part 4: Applications to Mathematical Physics.
\newblock {\em Springer}, 1997.

\end{thebibliography}
\end{document}